\documentclass[12pt]{article}
\usepackage{amssymb}
\usepackage{amsmath}
\usepackage{amsfonts}
\usepackage[T1]{fontenc}

\usepackage{graphicx}
\usepackage{float}

\newtheorem{theorem}{Theorem}
\newtheorem{acknowledgement}[theorem]{Acknowledgement}

\newtheorem{corollary}[theorem]{Corollary}

\newtheorem{definition}{Definition}

\newtheorem{lemma}{Lemma}

\newtheorem{remark}{Remark}

\newenvironment{proof}[1][Proof]{\noindent\textbf{#1.} }{\ \rule{0.5em}{0.5em}}
\begin{document}

\begin{center}
\textbf{{\Large {Identities and relations related to the numbers of special words derived from special series with Dirichlet convolution}}}

\bigskip

{\Large %\textit{\Large{Canada}}
}

Irem KUCUKOGLU$^{1,a}$ and Yilmaz SIMSEK$^{1,b}$

\bigskip

$^{1}$Department of Mathematics, Faculty of Science Akdeniz University of
Akdeniz TR-07058 Antalya, Turkey

E-mail address: $^{a}$iremkucukoglu90@gmail.com, $^{b}$ysimsek@akdeniz.edu.tr

\bigskip
\end{center}

\textbf{{\Large {Abstract }}}\\

The aim of this paper is to define some new number-theoretic functions including necklaces polynomials and the numbers of special words such as Lyndon words. By using Dirichlet convolution formula with well-known number-theoretic functions, we derive some new identities and relations associated with Dirichlet series, Lambert series, and also the family of zeta functions including the Riemann zeta functions and polylogarithm functions.
By using analytic (meromorphic) continuation of zeta functions, we also derive identities and formulas including Bernoulli numbers and Apostol-Bernoulli numbers. Moreover, we give relations between number-theoretic functions and the Fourier expansion of the Eisenstein series. Finally, we give some observations and remarks on these functions.



\section{Introduction, Definitions and Notations}

In order to define new number-theoretic functions including necklaces polynomials and the numbers of Lyndon words, we need the following definitions and notations:

In the following, let $\mathbb{Z}$, $\mathbb{R}$ and $\mathbb{C}$ be  the set of integers, the set of real numbers and the set of complex numbers, respectively. Let $\mathbb{Z}^{-}=\left\{ -1,-2,-3,...\right\}$, $\mathbb{Z}_{0}^{-}=\mathbb{Z}^{-}\cup \left\{ 0\right\}$, $\mathbb{N}=\{1,2,\dots\}$ and $\mathbb{N}_0=\mathbb{N}\cup\{0\}$. Let $\operatorname{Re}\left(s\right)$ denote the real part of $s\in\mathbb{C}$.

The necklaces polynomials $N_{k}\left( n\right) $ counts the number of
necklaces consisting of $n$-coloured beads with $k$-distinct colours, given
by the following formula (\textit{cf}. \cite{BerstelPerrin2007}):
\begin{equation}
N_{k}\left( n\right) =\frac{1}{n}\sum_{d|n}\phi \left( \frac{n}{d}\right)
k^{d}
\label{NeckNum}
\end{equation}%
where $\phi \left( n\right) $ denotes the number-theoretic Euler totient
function.

Let $k,n\in \mathbb{N}$. Then, the number of the Lyndon words are given as follows:
\begin{equation}
L_{k}\left( n\right) =\frac{1}{n}\sum_{d|n}\mu \left( \frac{n}{d}\right)
k^{d}  \label{LynNum}
\end{equation}
(\textit{cf}. \cite{BerstelPerrin2007}, \cite{KucukogluMJOM}, \cite{Lothaire1997}). Here, $\mu $ denotes the M\"{o}bius function defined by
\begin{equation*}
\mu \left( n\right)=
\begin{cases}
1 & \text{if $n=1$,} \\
\left(-1\right)^r & \text{if $n$ is a product of $r$ distinct primes,} \\
0 & \text{if $n$ is not square-free}
\end{cases}
\end{equation*}
(\textit{cf}. \cite{Apostol1998}).

The numbers $L_{k}\left( n\right) $ appear in many combinatorial problems. According to Betten et al. \cite{BettenEtAl}, these numbers is equal to the number of orbits of any $n$-th order cyclic group of maximal length. The formula (\ref{LynNum}) also calculates the number of monic irreducible  polynomials of degree $n$ over Galois field $GF(k)$ (\textit{cf}. \cite{BettenEtAl}, \cite{Buchanan}).

It is well-known that Lyndon words are the lexicographically (dictionary order) smallest element of the set of conjugate class which is the result of cyclic shifts of the letters in a primitive word. Figure 1 shows how special binary words are represented by the periodic and primitive necklaces consisting of $6$-coloured beads with $2$-distinct colours.

\begin{figure}[H]
	\centering
	\includegraphics[width=0.5\textwidth]{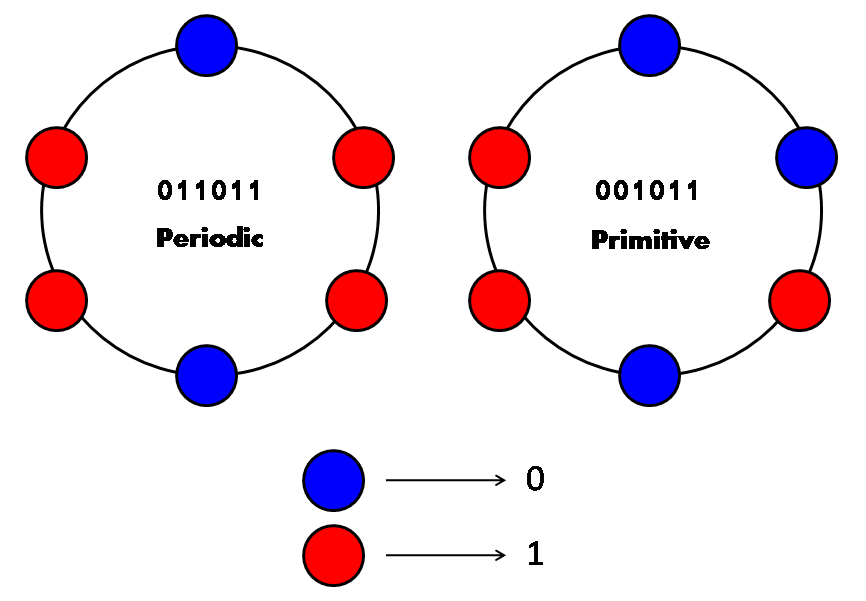}
	\caption{Representing special binary words by a periodic or a primitive necklaces.}
	\label{Representations}
\end{figure}

In Figure 1, we observe that the word $001011$ is primitive and smallest in its conjugacy class. Hence it is a Lyndon word.

\subsection{Number-theoretic functions and their algebraic properties}
Here, we recall algebraic properties of number-theoretic functions which are used in next sections.

Let $f\left(n\right)$ and $g\left(n\right)$ be number-theoretic (arithmetical) functions.

The well-known Dirichlet convolution of the functions $f\left( n\right) $ and $g\left( n\right) $ is given by
\begin{equation*}
\left(f*g\right)\left(n\right)=\sum_{d|n}f\left( d\right) g\left( \frac{n}{d}\right).
\end{equation*}
It is clear that the set of all number-theoretic functions is a commutative ring with unity under pointwise addition and Dirichlet convolutions of these functions (\textit{cf}. \cite{Delany2005}).

We now consider some properties of Dirichlet series, associated with the number-theoretic functions $f\left(n\right)$ and $g\left(n\right)$, which are defined as follows:
\begin{equation}
F\left(s\right)=\sum_{n=1}^{\infty}\frac{f\left(n\right)}{n^s} \quad \text{and} \quad G\left( z\right) =\sum_{n=1}^{\infty }\frac{g\left( n\right) }{n^{z}};
\label{DSeries1}
\end{equation}
these functions converge absolutely for the half-plane $\operatorname{Re}\left(s\right)>a$ and $\operatorname{Re}\left(z\right)>b$ where $a,b\in \mathbb{R}$ denote the abscissas of absolute convergence (\textit{cf}. \cite{Apostol1998}).

These functions are one of the most important useful tools in analytic number theory, probability theory and mathematical physics (\textit{cf}. \cite{Apostol1998}).

For $s$ is an element of the half-plane where both of the series in (\ref{DSeries1}) converge absolutely, the multiplication of $F\left( s\right)$ and $G\left( s\right)$ is given as follows:
\begin{equation}
F\left( s\right) G\left( s\right) =\sum_{n=1}^{\infty }\frac{h\left(
	n\right) }{n^{s}}  \label{DrichletProduct}
\end{equation}%
where $h\left( n\right) $ denotes %
\begin{equation*}
h\left( n\right) =\left(f*g\right)\left(n\right)
\end{equation*}%
(\textit{cf}. \cite[Theorem 11.5, p. 228]{Apostol1998}).

The Lambert series is also associated with the number-theoretic function $f\left(n\right)$, Dirichlet series and also Eisenstein series. The Lambert series is defined by
\begin{equation}
\sum_{n=1}^{\infty }f\left(n\right) \frac{x^{n}}{1-x^{n}}
\label{LSeries}
\end{equation}
and if we assume that (\ref{LSeries}) converges absolutely, then the following holds true
\begin{equation}
\sum_{n=1}^{\infty }f\left(n\right) \frac{x^{n}}{1-x^{n}}=\sum_{n=1}^{\infty }H\left(n\right)x^n
\label{LSeries1}
\end{equation}
where
\begin{equation*}
H\left(n\right)=\sum_{d|n}f\left(d\right) 
\end{equation*}
(\textit{cf}. \cite[p. 24]{Apostol1976}).

Substituting $f\left(n\right)=L_k\left(n\right)$ into (\ref{LSeries1}), after some elementary algebraic computations, for $\left\vert kx\right\vert <1$, one easily gets the following well-known Lambert series which is used in next section.

\begin{equation}
\sum_{n=1}^{\infty }nL_{k}\left( n\right) \frac{x^{n}}{1-x^{n}}=\frac{kx}{%
	1-kx}.  \label{Lambert-Lkn}
\end{equation}

Let $\mathbb{H}=\left\{ z\in \mathbb{C}:\operatorname{Im}\left( z\right) >0\right\} $. For $z\in \mathbb{H}$ and $\operatorname{Re}\left(s\right)>2$, the Eisenstein series $G\left(
z,k,r,h\right) $ is defined as follows (\textit{cf}. \cite{Apostol1976}, \cite{Lewittes}, \cite[Eq. (1.3)]{SimsekASCM2004}%
):%
\begin{equation*}
G\left( z,k,r,h\right) =\sum_{r\neq \left(m,n\right)\in \mathbb{Z}^2} \frac{e^{2\pi i \left(mh_1+nh_2\right)}}{\left(\left(m+r_1\right)z+n+r_2\right)^s}.
\end{equation*}

The Fourier expansion of the $G\left( z,k,r,h\right) $ is given as follows (\textit{cf}. \cite[Corollary 1]{SimsekASCM2004}):
\begin{equation}
G\left( z,k,r,h\right) =2\mathcal{Z} \left( k,h\right) +\frac{2\left( -2\pi
	i\right) ^{k}}{\left( k-1\right) !}\sum_{a,n=1}^{\infty }a^{k-1}e^{2\pi
	i\left( n+r\right) az},  \label{FourierExpG}
\end{equation}%
where $k\in \mathbb{N} \setminus\left\{ 1\right\} $, $r$ and $h$ are rational
numbers, $z\in \mathbb{H}$, and the function $\mathcal{Z} \left( k,h\right)$ is
\begin{equation*}
\mathcal{Z} \left( k,h\right)=\sum_{n>-a} \left(n+a\right)^{-s},
\end{equation*}
where $a\in \mathbb{R}$ and $\operatorname{Re}\left(s\right)>1$ (\textit{cf}. \cite{Lewittes}).

\subsection{Apostol-type numbers and polynomials with their interpolation functions}
Here, we recall some properties of generating functions for the Apostol-type numbers and polynomials, and also their interpolation functions.
 
The Apostol-Bernoulli polynomials $\mathcal{B}_{n}\left(x;\lambda\right) $ are
defined by%
\begin{equation}
\frac{te^{xt}}{\lambda e^{t}-1}=\sum_{n=0}^{\infty }\mathcal{B}_{n}\left(
x;\lambda\right) \frac{t^{n}}{n!}  \label{GF-ApostBern-Poly}
\end{equation}%
where $\lambda \in \mathbb{C}$; $\left\vert t\right\vert <2\pi $ when $\lambda=1$; $%
\left\vert t\right\vert <\left\vert \log \lambda\right\vert $ when $\lambda\neq 1$ (%
\textit{cf}. \cite{Apostol1951}). For $\lambda=1$, these polynomials are reduced to the generating function for Bernoulli polynomials $B_{n}\left(x\right)$ as follows:
\begin{equation}
\frac{te^{xt}}{e^{t}-1}=\sum_{n=0}^{\infty }B_{n}\left(x\right)\frac{t^{n}}{n!},
\label{GF-Bern-Poly}
\end{equation}%
(\textit{cf}. \cite{Apostol1951}, \cite{Boyadzhiev}, \cite{ChoiEtal2005}, \cite{ChoiEtal2008}, \cite{KucukogluMJOM}, \cite{Ozdenetal2010}, \cite{SrivastavaChoi}, \cite{SrivastavaChoi2012}, \cite{KucukogluJNT}).

Upon setting $x=0$ in (\ref{GF-ApostBern-Poly}), Apostol-Bernoulli polynomials are reduced to the Apostol-Bernoulli
numbers $\mathcal{B}_{n}\left(\lambda\right) $, given by%
\begin{equation}
\frac{t}{\lambda e^{t}-1}=\sum_{n=0}^{\infty }\mathcal{B}_{n}\left(\lambda\right) \frac{%
	t^{n}}{n!}  \label{GF-ApostBern-Num}
\end{equation}%
which, for $\lambda=1$, is reduced to the generating function for Bernoulli numbers $B_{n}$ as follows:
\begin{equation}
\frac{t}{e^{t}-1}=\sum_{n=0}^{\infty }B_{n}\frac{t^{n}}{n!},
\label{GF-Bern-Num}
\end{equation}%
which is related to Bernoulli polynomials by the following relation:
\begin{equation}
{B}_{m}\left( x\right) =\sum\limits_{j=0}^{m}
\binom{m}{j} x^{m-j}{B}_{j}
\label{Relation-BernoulliNumPoly}
\end{equation}%
and for $m>1$, the following holds true:
\begin{equation}
{B}_{m}\left( 1\right)={B}_{m}.
\label{R-BernNum-BernPoly}
\end{equation}
(\textit{cf}. \cite{Apostol1951}, \cite{Boyadzhiev}, \cite{ChoiEtal2005}, \cite{ChoiEtal2008}, \cite{KucukogluMJOM}, \cite{Ozdenetal2010}, \cite{SrivastavaChoi}, \cite{SrivastavaChoi2012}, \cite{KucukogluJNT})

Note that first few values of the Apostol-Bernoulli numbers are given as follows:
\begin{eqnarray*}
	\mathcal{B}_{0}\left(\lambda\right) &=&0, 	\mathcal{B}_{1}\left( \lambda\right) =\frac{1}{\lambda-1},
	\mathcal{B}_{2}\left(\lambda\right) =\frac{-2\lambda}{\left( \lambda-1\right) ^{2}}, \\
	\mathcal{B}_{3}\left( \lambda\right) &=&\frac{3\lambda\left( \lambda+1\right) }{\left(
		\lambda-1\right) ^{3}},	\mathcal{B}_{4}\left(\lambda\right) =\frac{-4\lambda\left( \lambda^{2}+4\lambda+1\right) }{\left(
		\lambda-1\right) ^{4}}, \ldots
\end{eqnarray*}%
and also first few values of the Bernoulli numbers are given as follows:
\begin{eqnarray*}
	B_{0}&=&1, B_{1}=-\frac{1}{2}, B_{2}=\frac{1}{6}, B_{3}=0, B_{4}=-\frac{1}{30}, B_{5}=0, B_{6}=\frac{1}{42},  B_{7}=0,\\
	&& B_{8}=-\frac{1}{30}, B_{9}=0, B_{10}=\frac{5}{66}, \ldots
\end{eqnarray*}%
with $B_{2n+1}=0, \left(n \in \mathbb{N}\right)$ (\textit{cf}. \cite{Apostol1951}, \cite{Boyadzhiev}, \cite{ChoiEtal2005}, \cite{ChoiEtal2008}, \cite{KucukogluMJOM}, \cite{Ozdenetal2010}, \cite{SrivastavaChoi}, \cite{SrivastavaChoi2012}, \cite{KucukogluJNT}).

For $m>1$, the relation between the Apostol-Bernoulli numbers and polynomials is
given by (\textit{cf}. \cite[p. 165, Eq-(3.5)]{Apostol1951}, \cite{Boyadzhiev}):
\begin{equation}
\lambda\mathcal{B}_{m}\left( 1;\lambda\right) =\mathcal{B}_{m}\left(\lambda\right).
\label{R-ApostBern-Bern}
\end{equation}

The Lerch transcendent function (or the Hurwitz-Lerch zeta function) $\Phi (\lambda,s,a)$ is given by (\textit{cf}. \cite[p.
121 et seq.]{SrivastavaChoi}):
\begin{eqnarray*}
	\Phi (\lambda,s,a) =\sum_{n=0}^{\infty }\frac{\lambda^{n}}{(n+a)^{s}},
\end{eqnarray*}%
which converges for the case of $a\in \mathbb{C\diagdown Z}_{0}^{-}$, $s\in \mathbb{C}$
when $\left\vert \lambda\right\vert <1$ and the case of $\operatorname{Re}\left(s\right)>1$ when $\left\vert \lambda\right\vert
=1$. For $m\in \mathbb{N}_{0}$, the function $\Phi (\lambda,s,a)$ is interpolation function of
Apostol-Bernoulli numbers, that is%
\begin{equation}
\Phi (\lambda,-m,0)=-\frac{\mathcal{B}_{m+1}\left(\lambda\right) }{m+1}.
\label{Int-Apost-Bern}
\end{equation}%
In the special case of $a=1$ and $\lambda=1$, the function $\Phi (\lambda,s,a)$ reduces to
the Riemann zeta function given by
\begin{equation*}
\Phi (1,s,1)=\zeta (s)=\sum_{n=1}^{\infty }\frac{1}{n^{s}}\text{, }\operatorname{Re}\left(s\right)>1
\end{equation*}%
which, for $s\rightarrow-m$ ($m\in \mathbb{N}_{0}$), is interpolation function of Bernoulli numbers, that is%
\begin{equation}
\zeta (-m)=-\frac{B_{m+1}}{m+1}
\label{Int-Bern}
\end{equation}%
(\textit{cf}. \cite{Apostol1951}, \cite{Boyadzhiev}, \cite{ChoiEtal2005}, \cite{ChoiEtal2008}, \cite{KucukogluMJOM}, \cite{Ozdenetal2010}, \cite{SrivastavaChoi}, \cite{SrivastavaChoi2012}, \cite{KucukogluJNT}).

It is well known that the polylogarithm function $Li_{s}(z)$ is given in terms of the Lerch transcendent function as follows (\textit{cf}. \cite{Guillera and Sondow}):
\begin{equation}
\lambda\Phi (\lambda,s,1)=Li_{s}(\lambda)=\sum_{n=1}^{\infty }\frac{\lambda^{n}}{n^{s}}
\label{Polylog-F}
\end{equation}
which, for $s\rightarrow-m$ ($m\in \mathbb{N}_{0}$), is interpolation function of the Eulerian numbers $%
A\left( m,k\right) $ and the Apostol-Bernoulli numbers. That is
\begin{equation}
Li_{-m}(\lambda)=\frac{1}{\left( 1-\lambda\right) ^{m+1}}\sum_{j=0}^{m}A\left(
m,j\right) \lambda^{m-j},
\label{Int-PolyLog}
\end{equation}%
and
\begin{equation}
Li_{-m}(\lambda)=-\frac{\mathcal{B}_{m+1}\left(\lambda\right) }{m+1}.
\label{Int-PolyLog2}
\end{equation}
(\textit{cf}. \cite{Comtet}, \cite{SrivastavaChoi2012}).

\section{Dirichlet type series involving new number-theoretic functions}

In this section, we define Dirichlet type series associated with new type number-theoretic functions related to the M\"{o}bius function and the Euler Totient phi functions. Relations between these series, the Riemann zeta functions and polylogarithm function are given. It is well-known that the Riemann zeta function and the polylogarithm function interpolates the Bernoulli numbers and the Apostol-Bernoulli numbers at negative integers, respectively.
In the light of this information, the values of the newly defined Dirichlet series at the negative integers have been examined. Moreover, some identities and relations involving Bernoulli numbers and the Apostol-Bernoulli numbers are given.

\subsection{New number-theoretic function including the numbers $L_{k}\left(n\right)$ and related Dirichlet series}

Here, by using (\ref{LynNum}), we set the following presumably new number-theoretic function:
	\begin{equation}
	L_{k}\left(x: n\right) =\frac{1}{n}\sum_{d|n}\mu \left( \frac{n}{d}\right)
	k^{d}x^d
	\label{Unification-LynNum}
	\end{equation}

\begin{remark}
	For $x=1$, \textup{(\ref{Unification-LynNum})} are reduced to \textup{(\ref{LynNum})}. That is
	\[L_{k}\left(n\right)=L_{k}\left(1: n\right).\]
\end{remark}

By selecting $n=6$, which is the smallest number that can be written as the multiplication of two distinct prime numbers, in (\ref{Unification-LynNum}), few values of the function $L_{k}\left(x: n\right)$ are given as follows:
\begin{eqnarray*}
L_{1}\left(x: 6\right)&=&\frac{x^6-x^3-x^2+x}{6},\\	
L_{2}\left(x: 6\right)&=&\frac{32x^6-4x^3-2x^2+x}{3}.
\end{eqnarray*}

By (\ref{Unification-LynNum}), we define the function $\zeta _{1}\left(x:k,s\right)$ Dirichlet series associated with the function $L_{k}\left(x: n\right)$ by the following definition:
\begin{definition}
	Let $k\in \mathbb{N}$ and  $x \in \mathbb{R}$. We define
	\begin{equation*}
	\zeta _{1}\left(x: k, s\right) =\sum_{n=1}^{\infty }\frac{nL_{k}\left(x: n\right)
	}{n^{s}}.
	\end{equation*}
\end{definition}

\begin{theorem}
	\label{Theorem-1}
	Let $k\in \mathbb{N}$ and $x \in \mathbb{R}$ such that $|kx|<1$. Then we have
	\begin{equation}
	\zeta (s)\zeta _{1}\left(x:k,s\right) =Li_{s}(kx)
	\label{Func-Zeta1}
	\end{equation}
	where $\operatorname{Re}\left(s\right)>1$.
\end{theorem}

\begin{proof}
	Let%
	\begin{equation*}
	F_{\mu }\left( s\right) =\sum_{n=1}^{\infty }\frac{\mu \left( n\right) }{%
		n^{s}}=\frac{1}{\zeta (s)},
	\end{equation*}%
	where $\operatorname{Re}\left(s\right)>1$ (\textit{cf}. \cite[p. 228, Example 1]{Apostol1998}) and
	\begin{equation*}
	G\left(x: k, s\right) =\sum_{n=1}^{\infty }\frac{\left(kx\right)^{n}}{n^{s}}=Li_{s}(kx),
	\end{equation*}
	where $|kx|<1$ and $\operatorname{Re}\left(s\right)>1 $.
	Replacing $F\left(s\right)$ by $F_{\mu }\left( s\right)$
	and $G\left(s\right)$ by $G\left(x: k, s\right)$ in (\ref{DrichletProduct}) yields the assertion of Theorem \ref{Theorem-1}.
\end{proof}

By substituting $s=-m$, with $m\in \mathbb{N}_0$, into (\ref{Func-Zeta1}) and by using (\ref{Int-Apost-Bern}) and (\ref{Int-Bern}), we  get the following theorem:
\begin{theorem}
	\label{Corollary-1}
	Let $k \in \mathbb{N}$. By beginning with
	 \begin{equation*}
	 \zeta _{1}\left(x: k, 0\right)=\frac{2 }{1-kx},
	 \end{equation*}
we have
\begin{equation}
\zeta _{1}\left(x: k, -m\right)=\frac{\mathcal{B}_{m+1}\left(kx\right) }{{B}_{m+1}}.
\label{Zeta1-Fun-Bernoulli}
\end{equation}
\end{theorem}

Here, by using generating functions method for the Bernoulli numbers and Apostol-Bernoulli numbers, we also give another proof of Theorem \ref{Corollary-1} as follows:

Substituting $f\left(n\right)=L_k\left(y:n\right)$ into (\ref{LSeries1}), after some elementary algebraic computations, for $\left\vert kyx\right\vert <1$, one easily gets the following well-known Lambert series which is used in next section.

\begin{equation}
\sum_{n=1}^{\infty }nL_{k}\left(y: n\right) \frac{x^{n}}{1-x^{n}}=\frac{kyx}{%
	1-kyx}.  \label{Lambert-Lkny}
\end{equation}
By substituting $x=e^{z}$ into (\ref{Lambert-Lkny}), we have%
\begin{equation*}
\sum_{n=1}^{\infty }nL_{k}\left(y: n\right) \frac{e^{zn}}{1-e^{zn}}=\frac{%
	kye^{z}}{1-kye^{z}}
\end{equation*}%
By using (\ref{GF-ApostBern-Num}) and (\ref{GF-Bern-Num}) in the above equation, we have%
\begin{equation*}
\sum_{n=1}^{\infty }nL_{k}\left(y: n\right) \sum_{m=0}^{\infty }\left(\frac{n^{m}}{m+1}\sum_{j=0}^{m+1}\binom{m+1}{j}B_{j}\right)
\frac{z^{m}}{m!}=\frac{ky}{z}\sum_{m=0}^{\infty }\mathcal{B}_{m}\left(
1;ky\right) \frac{z^{m}}{m!}.
\end{equation*}%
Combining (\ref{Relation-BernoulliNumPoly}) with the above equation yields
\begin{equation*}
\sum_{m=0}^{\infty }\sum_{n=1}^{\infty } \frac{n^{m+1}L_{k}\left(y: n\right)B_{m+1}\left(1\right)}{m+1}
\frac{z^{m}}{m!}=ky\sum_{m=0}^{\infty }\frac{\mathcal{B}_{m+1}\left(
	1;ky\right)}{m+1} \frac{z^{m}}{m!}
\end{equation*}%
Comparing the coefficients of $\frac{z^{m}}{m!}$ on both sides of the above
equation and using (\ref{R-BernNum-BernPoly}) and (\ref{R-ApostBern-Bern}) yields the assertion of Theorem \ref{Corollary-1}.

By using (\ref{Int-Apost-Bern}), we also have the following theorem:
\begin{theorem}
	Let $k \in \mathbb{N}$. By beginning with
	\begin{equation*}
	\zeta _{1}\left(x: k, 0\right)=\frac{2 }{1-kx},
	\end{equation*}
	we have
	\begin{equation}
	\zeta _{1}\left(x: k, -m\right) =-\frac{m+1}{\left( 1-kx\right) ^{m+1}B_{m+1}}%
	\sum_{j=0}^{m}A\left( m,j\right) \left(kx\right)^{m-j}.
	\label{Zeta1-Fun-Eulerian}
	\end{equation}
\end{theorem}

\begin{remark}
	By comparing \textup{(\ref{Zeta1-Fun-Bernoulli})} and \textup{(\ref{Zeta1-Fun-Eulerian})}, we also derive preasumably known relation between the Eulerian numbers and Apostol-Bernoulli numbers as follows:
		\begin{equation*}
		\mathcal{B}_{m+1}\left(kx\right) =-\frac{m+1}{\left( 1-kx\right) ^{m+1}}%
		\sum_{j=0}^{m}A\left( m,j\right)\left(kx\right)^{m-j}.
		\end{equation*}
\end{remark}

Moreover, we construct another function, which interpolates Bernoulli numbers and Apostol-Bernoulli numbers, as follows:
\begin{equation*}
\zeta _{1, \text{odd}}\left(x: k, s\right) =\sum_{n=1}^{\infty }\frac{L_{k}\left(x:
	2n-1\right) }{\left( 2n-1\right) ^{s-1}}.
\end{equation*}

\begin{theorem}
	\label{Theorem-3}
		Let $k\in \mathbb{N}$ and $x \in \mathbb{R}$ such that $|kx|<1$. Then we have
	\begin{equation}
	\zeta (s)\zeta _{1, \text{odd}}\left(x: k, s\right)=\frac{2^sLi_{s}\left(kx\right)-Li_{s}\left(k^2x^2\right)}{2^s-1}.
	\label{Func-Zeta1odd}	
	\end{equation}
\end{theorem}

Before giving proof of Theorem \ref{Theorem-3}, we need the following lemmas:

\begin{lemma}
	If $m\in \mathbb{N}$ and $p$'s are prime numbers, then the following holds true%
	\begin{equation}
	\sum_{n=1}^{\infty }\frac{\mu \left( mn\right) }{n^{s}}=\frac{\mu \left(
		m\right) }{\zeta (s)}\prod\limits_{p|m}\frac{1}{1-p^{-s}},
	\label{LemmaMobius2n}
	\end{equation}
	(\textit{cf}. \textup{\cite[Eq. (2.31)]{Mobius2n2016Laohakosol}, \cite[p. 30]{Navas}}).
\end{lemma}

\begin{lemma}
	Let $k\in \mathbb{N}$ and $x \in \mathbb{R}$ such that $|kx|<1$. Then we have
	\begin{equation}
	\zeta _{1, \text{odd}}\left(x: k, s\right)=\sum_{n\not\equiv 0 \left(2\right)}\frac{\mu \left( n\right) }{%
		n^{s}}\sum_{m\not\equiv 0\left(2\right)}\frac{{\left(kx\right)}^{m}}{m^{s}}.
	\label{Zeta1-DirichletMulty}
	\end{equation}
	\label{Lemma1}
\end{lemma}

\begin{proof}
	In order to multiply the following series
	\begin{equation}
	\sum_{n\not\equiv 0 \left(2\right)}\frac{\mu \left( n\right) }{%
		n^{z}}
	\label{Partial-Mobius}
	\end{equation}
	and
	\begin{equation}
	\sum_{m\not\equiv 0\left(2\right)}\frac{{\left(kx\right)}^{m}}{m^{s}},
	\label{Partial-PolyLog}
	\end{equation}
	 we assume that $\operatorname{Re}\left(z\right)>b_1$; and also $|kx|<1$ and $\operatorname{Re}\left(s\right)>b_2$, respectively. By the light of proof of Theorem 11.5 in \cite{Apostol1998}, for the half-plane in which both series (\ref{Partial-Mobius}) and (\ref{Partial-PolyLog}) are absolute convergence, we have
	\begin{equation*}
	\sum_{n\not\equiv 0 \left(2\right)}\frac{\mu \left( n\right) }{%
		n^{s}}\sum_{m\not\equiv 0\left(2\right)}\frac{{\left(kx\right)}^{m}}{m^{s}}=\sum_{n\not\equiv 0 \left(2\right)}\sum_{m\not\equiv 0 \left(2\right)}\mu \left( n\right)\left(kx\right)^m\left(nm\right)^{-s}.
	\end{equation*}
	Setting $nm=a$; $a=1,3,5,\dots$ in the above equation and using Dirichlet convolution formula, we deduce that
	\begin{eqnarray*}
		\sum_{n\not\equiv 0 \left(2\right)}\frac{\mu \left( n\right) }{%
			n^{s}}\sum_{m\not\equiv 0\left(2\right)}\frac{{\left(kx\right)}^{m}}{m^{s}}&=&\sum_{a\not\equiv 0 \left(2\right)}\sum_{mn=a}\mu \left( n\right)\left(kx\right)^m a^{-s}\\
		&=&\sum_{a\not\equiv 0 \left(2\right)}\frac{\sum_{d|a}\mu \left(\frac{a}{d}\right)\left(kx\right)^d}{a^s},
	\end{eqnarray*}
	which yields the assertion of Lemma \ref{Lemma1}.
\end{proof}

\begin{proof}[Proof of Theorem \ref{Theorem-3}]
Substituting $m=2$ into (\ref{LemmaMobius2n}), we have%
\begin{equation*}
\sum_{n=1}^{\infty }\frac{\mu \left( 2n\right) }{n^{s}}=-\frac{1}{\zeta (s)}%
\left( \frac{1}{1-2^{-s}}\right) .
\end{equation*}
Thus, we have
\begin{equation*}
F_{\mu, \text{odd} }\left( s\right) =\sum_{n\not\equiv 0 \left(2\right)}\frac{\mu \left( n\right) }{%
	n^{s}}=\frac{1}{\zeta (s)}\left(\frac{2^s}{2^s-1}\right).
\end{equation*}%
Moreover, we also have
\begin{eqnarray*}
	G_{\text{odd} }\left(kx, s\right)& =&\sum_{n\not\equiv 0\left(2\right)}\frac{{\left(kx\right)}^{n}}{n^{s}}\\
	&=&Li_{s}\left(kx\right)-\sum_{n\equiv 0\left(2\right)}\frac{{\left(kx\right)}^{n}}{n^{s}}\\
	&=&Li_{s}\left(kx\right)-\frac{1}{2^s}Li_{s}\left(k^2x^2\right).
\end{eqnarray*}
Replacing $F\left(s\right)$ by $F_{\mu, \text{odd} }\left( s\right)$
and $G\left(s\right)$ by $G_{\text{odd} }\left(kx, s\right)$ in (\ref{DrichletProduct}) and combining the final equation with (\ref{Zeta1-DirichletMulty}), we get the assertion of Theorem \ref{Theorem-3}.
\end{proof}

By substituting $s=-m$, with $m\in \mathbb{N}$, into (\ref{Func-Zeta1odd}) and by using (\ref{Int-Apost-Bern}) and (\ref{Int-Bern}), we  get the following theorem:
\begin{theorem}
	Let $k\in \mathbb{N}$. Then we have
	\begin{equation}
	\zeta _{1, \text{odd}}\left(x: k, -m\right)=\frac{\mathcal{B}_{m+1}\left(kx\right) -2^m\mathcal{B}_{m+1}\left(k^2x^2\right)}{\left(1-2^m\right){B}_{m+1}}.
	\end{equation}
\end{theorem}

Since the function $\zeta _{1}\left(x:k,s\right)$ converges absolutely for $\operatorname{Re}\left(s\right)>1$ and $|kx|<1$, we also set
\begin{eqnarray}
\zeta _{1, \text{even}}\left(x: k, s\right)&=&\sum_{n=1}^{\infty }\frac{L_{k}\left(x:2n\right) }{\left( 2n\right) ^{s-1}} \notag\\
&=&\zeta _{1}\left(x: k, s\right)-\zeta _{1, \text{odd}}\left(x: k, s\right).
\label{Func-Zeta1Even}
\end{eqnarray}

By combining (\ref{Func-Zeta1}) and (\ref{Func-Zeta1odd}) with (\ref{Func-Zeta1Even}), we have the following theorem:
\begin{theorem}
	Let $k\in \mathbb{N}$ and $x \in \mathbb{R}$ such that $|kx|<1$. Then we have
	\begin{equation}
	\zeta (s)\zeta _{1, \text{even}}\left(x: k, s\right)=\frac{Li_{s}\left(k^2x^2\right)-Li_{s}\left(kx\right)}{2^s-1}.
	\label{Func-Zeta1Even-Th}	
	\end{equation}
\end{theorem}

By substituting $s=-m$, with $m\in \mathbb{N}$, into (\ref{Func-Zeta1odd}) and by using (\ref{Int-Apost-Bern}) and (\ref{Int-Bern}), we  get the following theorem:
\begin{theorem}
	Let $k \in \mathbb{N}$. Then we have
	\begin{equation}
	\zeta _{1, \text{even}}\left(x: k, -m\right)=\frac{2^m\left(\mathcal{B}_{m+1}\left(k^2x^2\right)-\mathcal{B}_{m+1}\left(kx\right)\right) }{\left(1-2^m\right){B}_{m+1}}.
	\end{equation}
\end{theorem}

\subsection{New number-theoretic function including the numbers $N_{k}\left(n\right)$ and related Dirichlet series}
Here, by using (\ref{NeckNum}), we set the following presumably new number-theoretic function related to the necklace polynomials:
	\begin{equation}
	N_{k}\left(x: n\right) =\frac{1}{n}\sum_{d|n}\phi \left( \frac{n}{d}\right)
	k^{d}x^d.
	\label{Unification-NeckNum}
	\end{equation}

\begin{remark}
	For $x=1$, \textup{(\ref{Unification-NeckNum})} are reduced to \textup{(\ref{NeckNum})}. That is
	\[N_{k}\left(n\right)=N_{k}\left(1: n\right).\]
\end{remark}

With the help of (\ref{Unification-NeckNum}), few values of the function $N_{k}\left(x: n\right)$ are given as follows:
\begin{eqnarray*}
	N_{1}\left(x: 6\right)&=&\frac{x^6+x^3+2x^2+2x}{6},\\	
	N_{2}\left(x: 6\right)&=&\frac{32x^6+4x^3+4x^2+2x}{3}.
\end{eqnarray*}

By (\ref{NeckNum}), we also define another Dirichlet series associated with the function $N_{k}\left(x: n\right)$ by the following definition:
\begin{definition}
	Let $k\in \mathbb{N}$, $x \in \mathbb{R}$ and $s \in \mathbb{C}$. We define
	\begin{equation*}
	\zeta _{2}\left(x: s,k\right) =\sum_{n=1}^{\infty }\frac{nN_{k}\left(x: n\right)
	}{n^{s}}.
	\end{equation*}
\end{definition}

\begin{theorem}
	\label{Theorem-2}%
	Let $k\in \mathbb{N}$, $x \in \mathbb{R}$ and $s \in \mathbb{C}$ such that $\operatorname{Re}\left(s\right)>2$ and $|kx|<1$. Then we have
	\begin{equation}
	\zeta (s)\zeta _{2}\left(x: s,k\right) =\zeta (s-1)Li_{s}(kx).
	\label{Func-Zeta2}
	\end{equation}
\end{theorem}

\begin{proof}
	Let
	\begin{equation}
	F_{\phi }\left( s\right) =\sum_{n=1}^{\infty }\frac{\phi \left( n\right) }{%
		n^{s}}=\frac{\zeta (s-1)}{\zeta (s)},
	\label{EulerPhi}
	\end{equation}%
	where $\operatorname{Re}\left(s\right)>2$ (\textit{cf}. \cite[Example 4]{Apostol1998}). Thus, substituting $F\left(s\right)=F_{\phi }\left( s\right)$
	and $G\left(s\right)=G\left(x: k, s\right)$ into (\ref{DrichletProduct}) yields the assertion of Theorem \ref%
	{Theorem-2}.
\end{proof}

By substituting $s=-m$, with $m\in \mathbb{N}$, into (\ref{Func-Zeta2}) and by using (\ref{Int-Apost-Bern}) and (\ref{Int-Bern}), we get the following theorem:
\begin{theorem}
	\label{Corollary-2}
	Let $k\in \mathbb{N}$. Then we have
	\begin{equation}
	\zeta _{2}\left(x: k, -m\right)=-\frac{B_m\mathcal{B}_{m+1}\left(kx\right) }{m{B}_{m+1}}.
	\label{Zeta2-Fun-Bernoulli}
	\end{equation}
\end{theorem}

Now, we construct another function, which interpolates the Bernoulli numbers and Apostol-Bernoulli numbers, as follows:
\begin{equation*}
\zeta _{2, \text{odd}}\left(x: k, s\right) =\sum_{n=1}^{\infty }\frac{N_{k}\left(x:
	2n-1\right) }{\left( 2n-1\right) ^{s-1}},
\end{equation*}
where $k\in \mathbb{N}$, $x \in \mathbb{R}$ and $s \in \mathbb{C}$.

\begin{theorem}
	\label{Theorem-4}
	Let $k\in \mathbb{N}$, $x \in \mathbb{R}$ and $s \in \mathbb{C}$ such that $\operatorname{Re}\left(s\right)>2$ and $|kx|<1$. Then we have
	\begin{equation}
	2^s\zeta\left(s\right)\zeta _{2, \text{odd}}\left(x: k, s\right)=\left(2^s-1\right)\zeta\left(s-1\right)\left(Li_{s}\left(kx\right)-\frac{1}{2^s}Li_{s}\left(k^2x^2\right)\right).
	\label{Func-Zeta2odd}	
	\end{equation}
\end{theorem}

Before giving proof of Theorem \ref{Theorem-4}, we need the following lemma:
\begin{lemma}
	Let $k\in \mathbb{N}$, $x \in \mathbb{R}$ and $s \in \mathbb{C}$ such that $\operatorname{Re}\left(s\right)>2$ and $|kx|<1$. Then we have
	\begin{equation}
	\zeta _{2, \text{odd}}\left(x: k, s\right)=
	\sum_{n\not\equiv 0 \left(2\right)}\frac{\phi \left( n\right) }{%
		n^{s}}\sum_{m\not\equiv 0\left(2\right)}\frac{{\left(kx\right)}^{m}}{m^{s}}.
	\label{Zeta2-DirichletMulty}
	\end{equation}
	\label{Lemma2}
\end{lemma}

\begin{proof}
In order to multiply (\ref{Partial-PolyLog}) with the following series 
	\begin{equation}
	\sum_{n\not\equiv 0 \left(2\right)}\frac{\phi \left( n\right) }{%
		n^{s}},
	\label{Partial-EulerTotient}
	\end{equation}
	we assume that $\operatorname{Re}\left(s\right)>b_1$. By the same technique in Lemma \ref{Lemma1}, for the half-plane in which both series (\ref{Partial-PolyLog}) and (\ref{Partial-EulerTotient}) are absolute convergence, we have
	\begin{equation*}
\sum_{n\not\equiv 0 \left(2\right)}\frac{\phi \left( n\right) }{%
		n^{s}}\sum_{m\not\equiv 0\left(2\right)}\frac{{\left(kx\right)}^{m}}{m^{s}}=\sum_{n\not\equiv 0 \left(2\right)}\sum_{m\not\equiv 0 \left(2\right)}\phi \left( n\right)\left(kx\right)^m\left(nm\right)^{-s}.
	\end{equation*}
	Setting $nm=a$; $a=1,3,5,\dots$ in the above equation and using Dirichlet convolution formula, we deduce that
	\begin{eqnarray*}
		\sum_{n\not\equiv 0 \left(2\right)}\frac{\phi \left( n\right) }{%
			n^{s}}\sum_{m\not\equiv 0\left(2\right)}\frac{{\left(kx\right)}^{m}}{m^{s}}&=&\sum_{a\not\equiv 0 \left(2\right)}\sum_{mn=a}\phi \left( n\right)\left(kx\right)^m a^{-s}\\
		&=&\sum_{a\not\equiv 0 \left(2\right)}\frac{\sum_{d|a}\phi \left(\frac{a}{d}\right)\left(kx\right)^d}{a^s},
	\end{eqnarray*}
	which yields the assertion of Lemma \ref{Lemma2}.
\end{proof}

\begin{proof}[Proof of Theorem \ref{Theorem-4}]
	We set
	\begin{equation*}
	F_{\phi, \text{odd} }\left( s\right) =\sum_{n\not\equiv 0 \left(2\right)}\frac{\phi \left( n\right) }{%
		n^{s}}.
	\end{equation*}%
	By using the well-known property of the Euler phi function $\phi\left(n\right)$ for $n$ even integer $\phi\left(2n\right)=2\phi\left(n\right)$ and $n$ odd integer $\phi\left(2n\right)=\phi\left(n\right)$, equation (\ref{EulerPhi}) can be written as follows:
	\begin{equation*}
	\sum_{n=1}^{\infty }\frac{\phi \left( 2n\right) }{\left(2n\right)^{s}}=\frac{\zeta (s-1)}{2^s\zeta (s)}.
	\end{equation*}
	By the above equation, we obtain
	\begin{equation*}
	F_{\phi, \text{odd} }\left( s\right)=\left(1-\frac{1}{2^s}\right)\frac{\zeta (s-1)}{\zeta (s)}.
	\end{equation*}%
	Substituting $F\left(s\right)=F_{\phi, \text{odd} }\left( s\right)$
	and $G\left(s\right)=G_{\text{odd} }\left(kx, s\right)$ into (\ref{DrichletProduct}) and combining the final equation with (\ref{Zeta2-DirichletMulty}), we get assertion of Theorem \ref{Theorem-4}.
\end{proof}

By substituting $s=-m$, with $m\in \mathbb{N}$, into (\ref{Func-Zeta2odd}) and by using (\ref{Int-Apost-Bern}) and (\ref{Int-Bern}), we  get the following theorem:
\begin{theorem}
	Let $k \in \mathbb{N}$. Then we have
	\begin{equation}
	\zeta _{2, \text{odd}}\left(x: k, -m\right)=
	\frac{\left(1-2^m\right)B_m\left(2^m\mathcal{B}_{m+1}\left(k^2x^2\right)-\mathcal{B}_{m+1}\left(kx\right)\right)}{mB_{m+1}}.
	\end{equation}
\end{theorem}

Now, we define
\begin{eqnarray*}
\zeta _{2, \text{even}}\left(x: k, s\right)&=&\sum_{n=1}^{\infty }\frac{N_{k}\left(x:2n\right) }{\left( 2n\right) ^{s-1}},
\end{eqnarray*}
where $k\in \mathbb{N}$, $x \in \mathbb{R}$ and $s \in \mathbb{C}$.

Since the function $\zeta _{2}\left(x:k,s\right)$ converges absolutely for $\operatorname{Re}\left(s\right)>2$ and $|kx|<1$, we set
\begin{eqnarray}
\zeta _{2, \text{even}}\left(x: k, s\right)=\zeta _{2}\left(x: k, s\right)-\zeta _{2, \text{odd}}\left(x: k, s\right).
\label{Func-Zeta2Even}
\end{eqnarray}
By combining (\ref{Func-Zeta2}) and (\ref{Func-Zeta2odd}) with (\ref{Func-Zeta2Even}), we have the following theorem:
\begin{theorem}
	Let $k\in \mathbb{N}$, $x \in \mathbb{R}$ and $s \in \mathbb{C}$ such that $\operatorname{Re}\left(s\right)>2$ and $|kx|<1$. Then we have
	\begin{equation}
	2^s\zeta\left(s\right)\zeta _{1, \text{even}}\left(x: k, s\right)=\zeta\left(s-1\right)\left(Li_{s}\left(kx\right)-\frac{1-2^s}{2^s}Li_{s}\left(k^2x^2\right)\right).
	\label{Func-Zeta2Even-Th}	
	\end{equation}
\end{theorem}

By substituting $s=-m$, with $m\in \mathbb{N}$, into (\ref{Func-Zeta1odd}) and by using (\ref{Int-Apost-Bern}) and (\ref{Int-Bern}), we  get the following theorem:
\begin{theorem}
	Let $k \in \mathbb{N}$. Then we have
	\begin{equation}
	\zeta _{1, \text{even}}\left(x: k, -m\right)=\frac{2^mB_m\left(\mathcal{B}_{m+1}\left(kx\right)-\left(2^m-1\right)\mathcal{B}_{m+1}\left(k^2x^2\right)\right)}{mB_{m+1}}.
	\end{equation}
\end{theorem}

Combining (\ref{Func-Zeta1}) and (\ref{Func-Zeta2}) yields the following corollary:
\begin{corollary}
	Let $k\in \mathbb{N}$, $x \in \mathbb{R}$ and $s \in \mathbb{C}$ such that $\operatorname{Re}\left(s\right)>2$ and $|kx|<1$. Then we have
	\begin{equation*}
	\zeta\left(s-1\right)\zeta _{1}\left(x: k, s\right)=\zeta _{2}\left(x: k, s\right).
	\end{equation*}
\end{corollary}

\section{Relations involving Lambert series and Eisenstein series}

In this section, by using Fourier expansion of the Eisenstein series, we give some identities for the Lambert series related to the number of the Lyndon words.

Now, for $n \in \mathbb{N}$, we define the following Lambert series
\begin{equation}
H\left( n,x\right) =n\sum_{k=1}^{\infty }L_{k}\left( n\right) \frac{x^{k}}{%
	1-x^{k}}=n\sum_{k,m=1}^{\infty }L_{k}\left( n\right) x^{km}. \label{Def-H}
\end{equation}%
Setting $x=e^{2\pi iz}$ into (\ref{Def-H}) yields%
\begin{equation*}
H\left( n,e^{2\pi iz}\right) =\sum_{k,m=1}^{\infty }nL_{k}\left( n\right)
e^{2\pi ikmz}.
\end{equation*}%
Combining the above equation with (\ref{LynNum}) yields the following
theorem:

\begin{theorem}
	\begin{equation}
	H\left( n,e^{2\pi iz}\right) =\sum_{d|n}\mu \left( \frac{n}{d}\right)
	\sum_{k,m=1}^{\infty }k^{d}e^{2\pi ikmz}.  \label{Result1}
	\end{equation}
\end{theorem}

By combining (\ref{Result1}) with (\ref{FourierExpG}), we get the following theorem:
\begin{theorem}
	\begin{equation*}
	H\left( n,e^{2\pi iz}\right) =\sum_{d|n}d!\mu \left( \frac{n}{d}\right) \frac{G\left( z,d+1,0,h\right) -2\mathcal{Z} \left( d+1,h\right)}{2\left( -2\pi i\right) ^{d+1}}.
	\end{equation*}
\end{theorem}

If we replace $n$ by prime number $p$, then we arrive at the following
corollary:

\begin{corollary}
	Let $p$ be a prime number. Then we have
	\begin{equation*}
	H\left( p,e^{2\pi iz}\right) =\frac{p!\left( G\left( z,p+1,0,h\right)
		-2\mathcal{Z} \left( p+1,h\right) \right) }{2\left( -2\pi i\right) ^{p+1}}+\frac{%
		G\left( z,2,0,h\right) -2\mathcal{Z} \left( 2,h\right) }{8\pi ^{2}}.
	\end{equation*}
\end{corollary}

\section{Conclusions}

In this paper, we define some new number-theoretic functions including necklaces polynomials and the numbers of special words such as Lyndon words. By using Dirichlet convolution formula with well-known number-theoretic functions, we derive some new identities and relations associated with Dirichlet series, Lambert series, and also the family of zeta functions including the Riemann zeta functions and polylogarithm functions.
By using analytic (meromorphic) continuation of zeta functions, we also derive identities and formulas including Bernoulli numbers and Apostol-Bernoulli numbers. Moreover, we give relations between number-theoretic functions and the Fourier expansion of the Eisenstein series. The results of this paper have potential to application to not only analytic number theory, but also mathematical physics, and also related areas.

\begin{acknowledgement}
This paper was supported by the Scientific Research Project Administration of Akdeniz University (with Project Number: FBA-2018-3292).
\end{acknowledgement}

\end{document}